\documentclass[a4paper,11pt]{amsart}

\usepackage[a4paper, margin=2cm]{geometry}

\usepackage[utf8]{inputenc}
\usepackage{amsfonts,amssymb}
\usepackage{amsmath}
\usepackage{graphicx}
\usepackage{array}
\usepackage{enumitem}
\usepackage{url}
\usepackage{bbm}
\usepackage{comment}
\usepackage{hyperref}
\usepackage{ifthen}
\usepackage{calc}
\usepackage{xargs}

\usepackage{comment}
\usepackage[usenames,dvipsnames,table]{xcolor}

\usepackage{empheq}
\numberwithin{equation}{section}
\usepackage{tikz}
\usetikzlibrary{cd,decorations.pathreplacing,calc,intersections,arrows,arrows.meta}
\tikzstyle{dot}=[shape=circle,draw,color=black,fill=black,inner sep=1pt]

\usepackage{amsthm}
\theoremstyle{plain}

\newtheorem{theorem}{Theorem}[section]
\newtheorem{lemma}[theorem]{Lemma}
\newtheorem{corollary}[theorem]{Corollary}
\newtheorem*{corollary*}{Corollary}

\newtheorem{proposition}[theorem]{Proposition}

\theoremstyle{definition}

\newtheorem*{definition*}{Definition}
\newtheorem{example}[theorem]{Example}

\theoremstyle{remark}

\newcommand{\R}{\mathbf{R}}
\newcommand{\C}{\mathbf{C}}

\newcommand{\Z}{\mathbf{Z}}

\newcommand{\cK}{\mathcal{K}}
\newcommand{\cM}{\mathcal{M}}
\newcommand{\cG}{\mathcal{G}}
\newcommand{\cN}{\mathcal{N}}
\newcommand{\cP}{\mathcal{P}}

\newcommand{\cH}{\mathcal{H}}

\newcommand{\defi}[1]{{\bf #1}}

\DeclareMathOperator{\val}{val}
\DeclareMathOperator{\supp}{supp}
\DeclareMathOperator{\tr}{tr}

\DeclareMathOperator{\Aut}{Aut}

\begin{document}

\title[]{Almost synchronous correlations and Tomita-Takesaki theory}
\date{\today}


\author[A.~Marrakchi]{Amine Marrakchi}
\address{Université de Lyon, ENSL, CNRS, France}
\email{amine.marrakchi@ens-lyon.fr}
\author[M.~De~La~Salle]{Mikael De La Salle}
\address{Université de Lyon, Université Lyon 1, CNRS, France}
\thanks{MdlS was funded by the ANR grant Noncommutative analysis on groups and quantum groups ANR-19-CE40-0002-01}
\email{delasalle@math.univ-lyon1.fr}
\begin{abstract} The aim of this note is to present a ``type III'' generalization of a distribution lemma of Connes. We then derive, following Vidick, consequences on infinite-dimensional quantum strategies for non-local games.
\end{abstract}
\maketitle

Non-local games are the central objects in the recent solution to
Connes' embedding problem \cite{MIPRE}. In fact, all the games
considered there have the additional property of being synchronous. Let us recall the vocabulary.

A \defi{synchronous game} \cite{MR3933765} is a tuple $\cG=(X,\nu,A,D)$ where $X$ and $A$ are finite sets, $\nu$ is a symmetric probability measure on $X\times X$ and $D \colon X\times X\times A\times A\to \{0,1\}$ is a symmetric function satisfying $D(x,x,a,b) = 1_{a=b}$ for every $x \in X$ and $a,b \in A$. If $\alpha \in (0,1)$, we say that the synchronous game $\cG$ is \defi{$\alpha$-synchronous} if $\nu(x,x) \geq \alpha \sum_y \nu(x,y)$ for every $x \in X$. 

A \defi{commuting strategy} for $\cG$ is a family of probability
distributions on $A \times A$ of the form $\cP_{x,y}(a,b) = \langle
p^x_a \tilde{p}^y_b \xi,\xi\rangle$ where $\xi$ is a unit vector in a
Hilbert space $\cH$, $(p^x_a)_{a \in A}$ and $(\tilde{p}^y_b)_{b \in
  A}$ are partitions of unity (PVM) on $\cH$ such that
$[p^x_a,\tilde{p}^y_b]=0$ for every $x,y,a,b$.  A \defi{synchronous
  strategy} \cite{MR3460238} for $\cG$ is a commuting strategy such that
$\cP_{x,x}(a,b) =0$ for every $x \in X$ and $a\neq b$. The
\defi{value} of the strategy $\cP$ at the game $\cG$ is
  \[ \val(\cG,\cP) = \int \sum_{a,b} D(x,y,a,b) \cP_{x,y}(a,b) d\nu(x,y).\]

  So, given a synchronous game, synchronous strategies are those
  strategies that maximize
  \[ \int_X \sum_{a,b} D(x,y,a,b)
  \cP_{x,y}(a,b) d\nu(x,x) = \int_X \sum_{a} \cP_{x,y}(a,a)
  d\nu(x,x).\] It is therefore natural to expect that synchronous games
  that have a value close to $1$ have a synchronous value close to
  $1$. The next result confirms this expectation.
  \begin{theorem}\label{thm:main_games} Let $\varepsilon,\alpha \in (0,1)$ and $\cG$ be an $\alpha$-synchronous game admitting a commuting strategy with value $\geq 1-\varepsilon$. Then $\cG$ admits a synchronous strategy with value $\geq 1-c (\varepsilon/\alpha)^{\frac 1 4}$.
  \end{theorem}
  When only finite-dimensionnal Hilbert spaces are allowed,
  Theorem~\ref{thm:main_games} was proved by Vidick
  \cite{MR4373849}. The generalization to infinite dimensions has also
  independently been obtained by Lin \cite{lin2023synchronous} by a
  different route.  

  Observe (see \cite[Theorem 5.5]{MR3460238}) that synchronous
  strategies are exactly the strategies of the form $\cP_{x,y}(a,b) =
  \tau(p^x_a p^y_b)$ for a family of partitions of unity $(p^x_a)_{a
    \in A}$ in a von Neumann algebra $\cN$ with a tracial state
  $\tau$. Therefore, the elementary statement of
  Theorem~\ref{thm:main_games} can be equivalently stated in von
  Neumann algebraic language as follows.
  \begin{theorem}\label{thm:main_gamesvN} Let $\varepsilon,\alpha \in (0,1)$ and $\cG$ be an $\alpha$-synchronous game. Assume that there is a von Neumann algebra $\cM \subset B(\cH)$, a unit vector $\xi \in \cH$ a two families of partitions of the unity $(p^x_a)_{a \in A} \subset \cM$ and $(\tilde{p}^x_a)_{a \in A} \subset \cM'$ such that
    \[   \int \sum_{a,b} D(x,y,a,b) \langle p^x_a \tilde{p}^y_b \xi,\xi\rangle d\nu(x,y) \geq 1-\varepsilon.\]
    Then there is a von Neumann algebra $\cN$ with a tracial state $\tau$ and a family of partitions of unity $(r^x_a)_{a \in A} \subset \cN$ such that
    \[   \int \sum_{a,b} D(x,y,a,b) \tau(r^x_a r^y_b) d\nu(x,y) \geq 1-c(\varepsilon/\alpha)^{\frac 1 4}.\]
  \end{theorem}
What is amusing to notice is that Theorem~\ref{thm:main_games} is very
elementary to state, but we do not see any elementary proof of
it. Both Lin's and our proofs rely on quite advanced von Neumann
algebra techniques (the Connes-Tomita-Takesaki theory). More precisely
our proof relies on a new ``type III'' generalization of a
distribution lemma of Connes \cite{MR454659}. We believe that this
result (Proposition~\ref{prop:Connes_inequality}) is of independent
interest. Lin does not rely on our type III generalization of Connes' Lemma, but shows that arbitrary commuting correlations can be approximated by correlations coming from finite von Neumann algebras, and then runs Vidick's proof in the tracial setting. Our proof has the feature that the von Neumann algebra $\cM$ is completely explicit~: $\cM$ can be taken as a finite corner of the \emph{core} of $\cM$ appearing in the Connes-Tomita-Takesaki theory (as the crossed product of $\cM$ by its modular flow), see Corollary \ref{cor:application_games}. In the particular case when $\cM$ is semi-finite, $\cN$ can even be taken as a finite corner of $\cM$, see Corollary~\ref{cor:application_games_semifinite}.
\section{Connes' lemma}
Let $\cM$ be a von Neumann algebra. There is a unique (up to unique isomorphism) tuple $(c(\cM),\tau,\theta,\iota)$ of a von Neumann algebra $c(\cM)$ with a faithful semifinite trace $\tau$, and continuous group homomorphism $\theta \colon \R \to \Aut(\cM)$ and an embedding $\iota\colon \cM \to c(\cM)$ satisfying the conditions $\iota(\cM) = \{x \in c(\cM) \mid  \theta_s(x)=x \forall s\in \R\}$ and
\begin{equation}\label{eq:tautheta}\tau \circ \theta_s = e^{-s} \tau \forall s \in \R.
\end{equation}
It is called the \defi{core} of $\cM$. See \cite[Theorem 3.3, 3.4 and 3.5]{MR1829246} or \cite[Theorem 6.11]{MR1943006}.

\begin{example}\label{example:core_semifinite} If $\cM = \C$, $c(\cM)$ can be realized as $L_\infty(\R)$ with translation action $\theta^\C_s f = f(\cdot - s)$ and the trace $\tau^\C(f) = \int f(s) e^{-s} ds$.  If $\cM$ carries a semifinite trace $\tau_0$, then $c(\cM) = c(\C) \otimes \cM$ with the natural action $\theta^\C_s \otimes 1$ and trace $\tau^\C \otimes \tau_0$. If $\cM$ is not semifinite, the construction is much more difficult, it is the accomplishment of the Tomita-Takesaki theory.
\end{example}

Haagerup's \defi{non-commutative $L_p$ space} $L_p(\cM)$ \cite{MR560633} is the space of $\tau$-measurable operators affiliated with $c(\cM)$ and satisfying $\theta_s(x) = e^{-\frac s p} x$ for all $s$. Then for $1 \leq p<\infty$, $\|x\|_p:= \tau( \chi_{(1,\infty)}(|x|))^{\frac 1 p}$ is a complete norm on $L_p(\cM)$, and $(L_2(\cM),\|\cdot\|_2)$ is a Hilbert space. 

The fundamental fact due to Haagerup that we will use is that the non-commutative $L_1$ space $L_1(\cM)$ is naturally identified with $\cM_*$. Explicitely, the set of positive elements $\varphi \in \cM_*$ is in bijection with the positive elements $x \in L_1(\cM)$, the bijection being characterized by
\[ \varphi\Big(\int_\R \theta_s(z)ds \Big) = \tau(z x)\]
for every positive $z \in c(\cM)$. We write $\varphi(1) = \tr(x)$, and the linear map $\tr$ extends to a linear map $L_1(\cM) \to \R$ that satisfies $\tr(xy)= \tr(yx)$ for every $1\leq p \leq \infty$ and $x \in L_p(\cM)$, $y \in L_{p'}(\cM)$. We will make several uses of the following lemma.
\begin{lemma}\label{lem:tr_HaagerupLp} Let $1 \leq p<\infty$. If $x \in L_p(\cM)$ and $y \in L_{p'}(\cM)$ are positive, then
\[ \tr(xy) = \tau(x^{1-p}\chi_{(1,\infty)}(x) y).\]
\end{lemma}
\begin{proof} 
  Define $z = x^{-p} \chi_{(1,\infty)}(x)$ with the convention $0^{-p} \chi_{(1,\infty)}(0)=0$. Then 
  we have  $\int_\R \theta_s(z) = \supp(x)$ (see the proof of \cite[Lemma 5]{terp} for the details). By the definition of $\tr$, we deduce
  \[ \tau(z x^{\frac 1 2} y x^{\frac 1 2}) = \tr(x^{\frac 1 2} y x^{\frac 1 2}).\]
The lemma follows by the trace property.
\end{proof}
A particular case of this lemma is that for $x \in L_p(\cM)$, $\|x\|_p = (\tr( |x|^p))^{\frac 1 p}$.

We warn the reader that they are two different non-commutative $L_p$ spaces involved: Haagerup's $L_p(\cM)$ that we have just defined, and the tracial space $L_p(c(\cM),\tau)$, which is the space of operators affiliated with $c(\cM)$ and satisfying $\tau( |x|^p)^{\frac 1 p} < \infty$. They are both made of operators affiliated with $c(\cM)$ (but we have $L_p(c(\cM),\tau) \cap L_p(\cM) = \{0\}$). To avoid confusion, we will write $\|x\|_{L_p(\cM)}$ for the first, and $\|x\|_{L_p(\tau)}$ for the second.

Our goal is precisely to compare these two non-commutative $L_p$
spaces for $p=2$ (this can be combined with results from
\cite{MR3299148} to obtain similar inequalities for other values of
$p$). This generalizes \cite[Lemma 1.2.6]{MR454659} to non-semifinite
von Neumann algebras.
\begin{proposition}\label{prop:Connes_inequality} For every $x,y \in L_2(\cM)^+$,
  \[ \|x-y\|_{L_2(\cM)}^2\leq  \|\chi_{(1,\infty)}(x) - \chi_{(1,\infty)}(y)\|_{L_2(\tau)}^2 \leq \|x-y\|_{L_2(\cM)} \|x+y\|_{L_2(\cM)}\]
\end{proposition}

A moment of thoughts and a look at Example~\ref{example:core_semifinite} reveals for, for $\cM$ semi-finite, Proposition~\ref{prop:Connes_inequality} is indeed an equivalent form of \cite[Lemma 1.2.6]{MR454659}.

For the proof, we start with the following key observation of Connes \cite[Proposition 1.1]{MR454659} that gives the existence of a joint spectral measure for two positive operators in a tracial von Neumann algebra.
\begin{proposition}[Connes] \label{prop:existence_nu} 
Let $\cM$ be a von Neumann algebra with a normal semifinite trace $\tau$. Let $x$ and $y$ be two unbounded positive operators affilitated with $\cM$. Let $\nu_x : A \mapsto \tau(\chi_A(x))$ and $\nu_y : B \mapsto \tau(\chi_B(y))$ be the spectral measures of $x$ and $y$ respectively. Suppose that $\nu_x|_{A_0}$ and $\nu_y|_{B_0}$ are $\sigma$-finite for some Borel subsets $A_0, B_0 \subset \R_+$. Then there exists a unique $\sigma$-finite measure $\nu$ on $(A_0 \times \R_+) \cup (\R_+ \times B_0)$ such that
$$ \tau(\chi_A(x) \chi_B(y))=\nu(A \times B)$$
for all Borel sets $A,B \subset \R_+$ such that $A \subset A_0$ or $B \subset B_0$. 

Then for every pair of Borel functions $f,g : \R_+ \rightarrow \R_+$ such that $f$ is supported on $A_0$ or $g$ is supported on $B_0$, we have
$$ \tau(f(x)g(y))=\int_{\R_+^2} f(u)g(v) d\nu(u,v).$$
\end{proposition}

\begin{proposition}\label{prop:existence_mu} Let $x,y$ be a pair of positive elements of $L_2(\cM)$. There exists a unique finite measure $\mu =\mu_{x,y}$ on $[0,1]$ such that for every pair of Borel functions $f,g : \R_+ \rightarrow \R_+$ satisfying $f(0)g(0)=0$, we have

$$ \tau(f(x) g(y)) = \int_0^1 \int_0^\infty  f\left(  \frac{\lambda}{\sqrt{r}} \right) g\left( \frac{1-\lambda}{\sqrt{r}} \right) dr d\mu(\lambda).$$

Moreover, we have
  \begin{align} \label{eq:norm_x}\|x\|_{L_2(\cM)}^2 &= \int \lambda^2 d\mu(\lambda),\\
   \label{eq:norm_y} \|y\|_{L_2(\cM)}^2& = \int (1-\lambda)^2 d\mu(\lambda),\\
   \label{eq:norm_chix-chiy} \| \chi_{(1,\infty)}(x) - \chi_{(1,\infty)}(y)\|_{L_2(\tau)}^2 & =\int |2\lambda - 1| d\mu(\lambda)\\
    \label{eq:scalar_in_L2}\langle x,y\rangle_{L_2(\cM)} & =  \int \lambda(1-\lambda) d\mu(\lambda).
\end{align}
\end{proposition}
\begin{proof}
We apply Proposition \ref{prop:existence_nu}. Note that the spectral measures $\nu_x$ and $\nu_y$ are finite on $[\varepsilon,+\infty)$ for every $\varepsilon > 0$. In particular, there are $\sigma$-finite on $\R_+^*$ (but not necessarily on $\R_+$ since we have $\nu_x(\{0\})=\infty$ if $\supp(x) \neq 1$).  Therefore, by Proposition \ref{prop:existence_nu}, there is a unique $\sigma$-finite measure $\nu$ on $\R_+^2 \setminus \{ (0,0) \}$ such that
  \begin{equation}\label{eq:def_nu} \tau(f(x)g(y))=\int f(u)g(v) d\nu(u,v)
    \end{equation}
  for every pair of Borel functions $f,g : \R_+ \rightarrow \R_+$ such that $f(0)g(0)=0$. In fact $\nu$ is even a Radon measure on the locally compact space $\R^2_+ \setminus \{ (0,0) \}$ since it is finite on $([\varepsilon,+\infty) \times \R_+) \cup (\R_+ \times [\varepsilon,+\infty))$ for every $\varepsilon > 0$.

By \eqref{eq:tautheta}, we have 
$$e^{-2s}\tau(f(x)g(y))=\tau( \theta_{2s}(f(x)g(y)))=\tau(f(e^{-s}x)g(e^{-s}y)$$
Since $\nu$ is uniquely characterized by \ref{eq:def_nu}, we get $d\nu(e^su,e^sv)=e^{-2s} d \nu(u,v)$ for all $s \in \R$.
Let $\rho$ be the Radon measure on $[0,1] \times \R^*_+$ obtained as the pull-back of $\nu$ by the homeomorphism $$[0,1] \times \R^*_+ \ni (\lambda,r) \mapsto \left(\frac{\lambda}{\sqrt{r}},\frac{1-\lambda}{\sqrt{r}}\right)  \in \R^2_+ \setminus \{ (0,0) \}.$$
Then we have $d\rho(\lambda, e^{s}r)=e^{s} d\rho(\lambda,r)$. Therefore, $\rho$ must be of the form $\rho = \mu \otimes dr$ for some finite measure $\mu$ on $[0,1]$. Then a change of variable in equation \ref{eq:def_nu} allows us to conclude that
$$ \tau(f(x) g(y)) = \int_0^1 \int_0^\infty  f\left(  \frac{\lambda}{\sqrt{r}} \right) g\left( \frac{1-\lambda}{\sqrt{r}} \right) dr d\mu(\lambda)$$
for every pair of Borel functions $f,g : \R_+ \rightarrow \R_+$ satisfying $f(0)g(0)=0$.

Now, we have to prove the second part. The equations \eqref{eq:norm_x}, \eqref{eq:norm_y} and \eqref{eq:norm_chix-chiy} are immediate. For example, taking $f=\chi_{(1,\infty)}$ and $g=1$, we obtain
\[ \|x\|_{L_2(\cM)}^2 = \tau(\chi_{(1,\infty)}(x))= \int_0^1 \int_0^\infty 1_{r<\lambda^2} dr d\mu(\lambda) = \int_0^1 \lambda^2 d\mu(\lambda).\]
This proves \eqref{eq:norm_x}. The equality \eqref{eq:norm_y} is proved in the same way. For \eqref{eq:norm_chix-chiy}, take $f=g=\chi_{(1,\infty)}$:
\begin{align*}
  \| \chi_{(1,\infty)}(x)-\chi_{(1,\infty)}(y)\|_{L_2(\tau)}^2 & =\int \int_0^\infty |1_{r<\lambda^2} - 1_{r<(1-\lambda)^2}|^2 dr d\mu(\lambda)\\
    & = \int_0^1 |\lambda^2 - (1-\lambda)^2| d\mu(\lambda) \\
    & = \int_0^1 |2\lambda - 1| d\mu(\lambda).
\end{align*}

The last equality \eqref{eq:scalar_in_L2} is slightly more involved. By Lemma~\ref{lem:tr_HaagerupLp}, if $f(t) = t^{-1} \chi_{(1,\infty)}(t)$, we obtain
\[  \langle x,y\rangle_{L_2(\cM)}  = \tr(xy) = \tau( f(x) y).\]
This is equal to
\begin{align*} \int_0^1 \int_0^\infty f\left( \frac \lambda {\sqrt r} \right) \frac{1-\lambda}{\sqrt r} d\mu(\lambda) &= \int_0^1 \int_0^{\lambda^2} \frac{\sqrt{r}}{\lambda} \frac{1-\lambda}{\sqrt r} dr d\mu(\lambda)\\
  & = \int_0^1 \lambda(1-\lambda) d\mu(\lambda).\qedhere
\end{align*}
\end{proof}

\begin{proof}[Proof of Proposition~\ref{prop:Connes_inequality}] By Proposition~\ref{prop:existence_mu},
  \begin{align*} \|x-y\|_{L_2(\cM)}^2 &= \|x\|_{L_2(\cM)}^2 + \|y\|_{L_2(\cM)}^2 - 2 \langle x,y\rangle_{L_2(\cM)}\\
    & = \int (\lambda^2 + (1-\lambda)^2 - 2 \lambda(1-\lambda)) d\mu(\lambda) = \int (1-2\lambda)^2 d\mu(\lambda).
  \end{align*}
  In the same way,
  \[ \|x+y\|_{L_2(\cM)}^2  = \int 1 d\mu(\lambda).\]
  
So the claimed inequality is equivalent to 
\[ \int (1-2\lambda)^2 d\mu(\lambda) \leq \int |2\lambda - 1| d\mu(\lambda) \leq \sqrt{ \int (1-2\lambda)^2 d\mu(\lambda) \int 1 d\mu(\lambda)}.\]
This first inequality is clear because $|1-2\lambda| \leq 1$ on $[0,1]$. The second inequality is the Cauchy-Schwarz inequality.
\end{proof}

A particular case of Proposition~\ref{prop:Connes_inequality} is the following. Here and in the whole note, we write $[x,y]$ for the commutator $xy-yx$. Also, a \defi{partition of the unity} or \defi{PVM} in $\cM$ is a finite family $(p_1,\dots,p_n)$ of self-adjoint projections that sum to $1$. A finite family $(p_1,\dots,p_n)$ of positive operators that sum to $1$ is called a \defi{POVM}.
\begin{corollary}\label{cor:Connes_partition_of_unity}
  Let $x \in L_2(\cM)^+$ of norm $1$, and $(p_1,\dots,p_n) \subset \cM$ be a partition of the unity. Let $q = \chi_{(1,\infty)}(x)$. We have
  \[ \sum_k \|[p_k,x]\|_{L_2(\cM)}^2 \leq \sum_k \|[p_k, q]\|_{L_2(\tau)}^2 \leq 2 \|x\|_{L_2(\cM)} \Big(\sum_k \|[p_k, x]\|_{L_2(\cM)}^2\Big)^{\frac 1 2}.\]
\end{corollary}
\begin{proof} Let $U = \sum_k e^{\frac{2i\pi k}{n}} p_k \in U(\cM)$. Then by the orthogonality of the characters of $\Z/n\Z$, we have
\[\frac{1}{n} \sum_{k=1}^n \|U^k x U^{-k} - x\|_{L_2(\cM)}^2 = \frac{1}{n} \sum_{k=1}^n \|[U^k, x]\|_{L_2(\cM)}^2 =\sum_k \|[p_k ,x]\|_{L_2(\cM)}^2.\]
and similarly
\[ \frac{1}{n} \sum_{k=1}^n \|U^k q U^{-k} - q\|_{L_2(\tau)}^2 =\sum_k \|[p_k, q]\|_{L_2(\tau)}^2.\]
So the corollary follows from Proposition~\ref{prop:Connes_inequality} and the Cauchy-Schwarz inequality, because $$\left(\frac{1}{n} \sum_{k=1}^n \|U^k x U^{-k} + x\|_{L_2(\cM)}^2\right)^{\frac 1 2} \leq 2 \|x\|_{L_2(\cM)}.$$
\end{proof}

\section{Applications to correlations}
Let us introduce some notation. If $\cM \subset B(\cH)$ is a von Neumann algebra and $\xi \in \cH$ is a unit vector, define $q_\xi \in c(\cM)$ by $q_\xi =  \chi_{(1,\infty)}(h_\xi)$, where $h_\xi \in L_1(\cM)$ is the element corresponding to the the vector state $\langle \cdot \xi,\xi\rangle$ on $\cM$ through the isomorphism between $\cM_*$ and $L_1(\cM)$ described previously. It is a trace $1$ projection, because $\tau(q_\xi) =  \tr(h_\xi) = \langle 1 \xi,\xi\rangle=1$.

  \begin{theorem}\label{thm:almost_synchronous_correlations} Let $\cM\subset B(\cH)$ be a von Neumann algebra, $\xi \in \cH$ be a unit vector.

    Let $X,A$ be finite sets, $\nu$ be a symmetric probability measure on $X\times X$ with marginal $\mu$. For every $x \in X$, let $(p^x_a)_{a \in A} \subset \cM$ and $(\tilde{p}^x_a)_{a \in A} \subset \cM'$ be partitions of the unity.

    If $\int \sum_{a \in A} \langle p^x_a \tilde{p}^x_a \xi,\xi \rangle d\mu(x) \geq 1-\delta$, then
    \[ \int \sum_{a,b \in A} \Big|\langle p^x_a \tilde{p}^y_b \xi,\xi\rangle - \tau( q_\xi p^x_a q_\xi p^y_b q_\xi)\Big| d\nu(x,y) \lesssim \delta^{\frac 1 4}.\]

    Moreover, for every $x\in X$ there is a partition of the unity $(r^x_a)_{a \in A} \subset q_\xi c(\cM) q_\xi$ such that 
    \[ \int \sum_{a,b \in A} \Big|\tr( p^x_a h^{\frac 1 2} p^y_b h^{\frac 1 2}) - \tau( r^x_a r^y_b)\Big| d\nu(x,y) \lesssim \delta^{\frac 1 4}.\]
  \end{theorem}
Given Theorem~\ref{prop:Connes_inequality}, the proof is essentially the same as the proof of \cite[Theorem 3.1]{MR4373849}, but we write the proof for completeness.

  We shall use the following elementary Hilbert space lemma.
  \begin{lemma}\label{lem:Hilbert_space} Let $\cK$ be a Hilbert space, $(\Omega,m)$ be a measure space and $\xi,\eta \in L_2(\Omega,m;\cK)$. Then
    \[ \int \big| \|\xi(\omega)\|^2 - \|\eta(\omega)\|^2\big| dm(\omega) \leq \|\xi - \eta\| \|\xi+\eta\|.\]
  \end{lemma}
  \begin{proof} For any vectors in a Hilbert space, the inequality 
    \[ \big| \|u\|^2 - \|v\|^2\big| \leq \|u-v\| \|u+v\|\]
    is easy: squaring both sides, it is equivalent to
    \[ \|u\|^4+\|v\|^4 - 2 \|u\|^2 \|v\|^2 \leq \|u\|^4+\|v\|^4 + 2 \|u\|^2 \|v\|^2 - 4 |\Re\langle u,v\rangle|^2,\]
    which is just the Cauchy-Schwarz inequality. Taking $u=\xi(\omega)$ and $v=\eta(\omega)$, integrating and using the Cauchy-Schwarz inequality in $L_2(\Omega,m)$, we deduce the  lemma. 
    \end{proof}
The next lemma is the particular case of Theorem~\ref{thm:almost_synchronous_correlations} for strategies that are symmetric.
  \begin{lemma}\label{lem:symmetric_correlation_close_to_tracial} Let $\cM\subset B(\cH)$ be a von Neumann algebra, $h \in L_1(\cM)_+$ of norm $1$, and $q = \chi_{(1,\infty)}(h) \in L_2(c(\cM))$.

    Let $X,A,\nu,\mu$ be as in Theorem~\ref{thm:almost_synchronous_correlations} and, for every $x \in X$, let $(p^x_a)_{a \in A} \subset \cM$ be a partition of the unity.

    If $\int \sum_{a \in A} \| [p^x_a,h^{\frac 1 2}]\|_{L_2(\cM)}^2 d\mu(x) \leq \delta$, then 
    \[ \int \sum_{a,b \in A} \Big|\tr( p^x_a h^{\frac 1 2} p^y_b h^{\frac 1 2}) - \tau( q p^x_a q p^y_b q)\Big| d\nu(x,y) \leq 4 \delta^{\frac 1 4}.\]

    Moreover, for every $x\in X$ there is a partition of the unity $(r^x_a)_{a \in A} \subset q c(\cM) q$ such that 
    \[ \int \sum_{a,b \in A} \Big|\tr( p^x_a h^{\frac 1 2} p^y_b h^{\frac 1 2}) - \tau( r^x_a r^y_b)\Big| d\nu(x,y) \leq 38 \delta^{\frac 1 4}.\]
  \end{lemma}
  \begin{proof} Let $\Omega=X\times X\times A \times A$ and $m$ the product of the measure $\nu$ and the counting measure on $A \times A$. Consider the norm $1$ elements $C_1,C_2,C_3,C_4$ of $L_1(\Omega,m)$ defined by
        \begin{align*}
      C_1(x,y,a,b) & =\tr( h^{\frac 1 2} p^y_b h^{\frac 1 2} p^x_a) = \| p^x_a h^{\frac 1 2}  p^y_b\|_{L_2(\cM)}^2,\\
      C_2(x,y,a,b) &= \tr( h p^x_a p^y_b p^x_a)  = \| h^{\frac 1 2} p^x_a p^y_b\|_{L_2(\cM)}^2,\\
      C_3(x,y,a,b) &= \tau( q p^x_a p^y_b p^x_a) = \| q p^x_a p^y_b\|_{L_2(\tau)}^2,\\
      C_4(x,y,a,b) & =\tau( q p^y_b q p^x_aq) =  \| p^x_aq p^y_b\|_{L_2(\tau)}^2.
    \end{align*}
We have $C_2=C_3$ by Lemma~\ref{lem:tr_HaagerupLp}. Moreover, by
Lemma~\ref{lem:Hilbert_space}, we can bound
    \begin{align*} \|C_1 - C_2\|_1^2 & \leq {4 \int \sum_{a,b} \|[p^x_a,h^{\frac 1 2}] p^{y}_b\|_{L_2(\cM)}^2 d\nu(x,y)}\\
      &= {4 \int \sum_a \| [p^x_a,h^{\frac 1 2}]\|_{L_2(\cM)}^2d\mu(x)}\leq {4\delta}.
    \end{align*}
    In the same way, by Lemma~\ref{lem:Hilbert_space} and Corollary~\ref{cor:Connes_partition_of_unity}, we obtain
\begin{equation}\label{eq:normC3C4} \|C_3-C_4\|_1^2\leq 4 \int \sum_a \| [p^x_a,q]\|_{L_2(\tau)}^2 d\mu(x) \leq 8\sqrt{\delta}. 
\end{equation}
We deduce
\[ \|C_1 - C_4\|_1= \|C_1-C_2 + C_3-C_4\|_1 \leq 2\sqrt{\delta} + 2\sqrt{2} \delta^{\frac 1 4}.\]
We obtain
\[ \|C_1 - C_4\|_1 \leq \min(2, 2\sqrt{\delta} + 2\sqrt{\sqrt{2\delta}}) \leq 4 \delta^{\frac 1 4},\]
which is precisely the first conclusion of the lemma.

For the second conclusion, we use an orthogonalization result for POVMs that originates from the work of Kempe and Vidick \cite{zbMATH06301159}. We use a form that applies to infinite dimensional space and with the right dependance from \cite{orthonormalisation}. We can rewrite the second inequality in \eqref{eq:normC3C4} as
\[ \int \sum_a \tau( (q p^x_a q)^2) d\mu(x)  \geq 1 - 8 \sqrt{\delta},\]
so by \cite[Theorem 1.2]{orthonormalisation} for every $x \in X$, there is a partition of the unity $(r^x_a)_{a \in A}$ in $q c(\cM) q$ such that
\[ \sum_a \tau( |q p^x_a q - r^x_a|^2) \leq 9(1- \tau( (q p^x_a q)^2)).\]
As a consequence,
\[ \int \sum_a \tau( |q p^x_a q - r^x_a|^2) d\mu(x) \leq 72 \sqrt{\delta}.\]
Using twice Lemma~\ref{lem:Hilbert_space} gives
\[ \int \sum_{a,b} |\tau( q p^y_b q p^x_aq) - \tau( r^y_b  r^x_a)|d\nu(x,y) \leq 4 \sqrt{72 \sqrt{\delta}} \leq 34 \delta^{\frac 1 4}.\qedhere\]
  \end{proof}
  \begin{proof}[Proof of Theorem~\ref{thm:almost_synchronous_correlations}] The idea is simple: the assumption of the theorem implies that the correlation $\langle p^x_a \tilde{p}^y_b \xi,\xi\rangle$ is close to the symmetric correlation $\tr( p^x_a h_\xi^{\frac 1 2} p^y_b h_\xi^{\frac 1 2})$, so we can apply Lemma~\ref{lem:symmetric_correlation_close_to_tracial}. We proceed in several steps.

We first observe that for every $y$ there is a POVM $(p'^y_b)_{b \in A} \subset \cM$ such that for every $x,a,b$,
    \begin{equation}\label{eq:standardform} \langle p^x_a \tilde{p}^y_b \xi,\xi\rangle = \tr(p^x_a h_\xi^{\frac 1 2} p'^y_b h_\xi^{\frac 1 2}\rangle.
    \end{equation}
    This is a standard fact from von Neumann algebras, that we recall for completeness. The map $\langle \cdot \tilde{p}^y_b \xi,\xi\rangle$ on $\cM$ is a positive element of $\cM_*$, so there is a unique positive $h(\xi,y,b) \in L_1(\cM)_*$ such that $\langle \cdot \tilde{p}^y_b \xi,\xi\rangle = \tr(\cdot h(\xi,y,b))$. The fact that $\tilde{p}^y_b \leq 1$ implies that $0 \leq h(\xi,y,b) \leq h_\xi$, that is there is a unique $p'^y_b \in c(\cM)$ such that $0 \leq p'^y_b \leq \supp(h_\xi)$ and $h(\xi,y,b) = h_\xi^{\frac 1 2} p'^y_b h_\xi^{\frac 1 2}$. We have $\theta_s(p'^y_b) = p'^y_b$, so $p'^y_b\in \cM$. By uniqueness, we have $h_\xi^{\frac 1 2}\sum_b p'^y_b h_\xi^{\frac 1 2}= h_\xi$, that is $\sum_b p'^y_b = \supp(h_\xi)$. Replacing $p'^y_b$ by $p'^y_b + (1-\supp(h_\xi))$ for some $b$ turns $p'^y$ to a POVM still satisfying \eqref{eq:standardform}.

We now claim that
\begin{equation}\label{eq:commutator_p_h} \int \sum_a \| [p^x_a,h_\xi^{\frac 1 2}]\|_{L_2(\cM)}^2 d\mu(x) \leq 4\delta.
\end{equation}
Before we justify this, observe that by Lemma~\ref{lem:symmetric_correlation_close_to_tracial}, this implies
\begin{equation}\label{eq:correlations_symmetric} \int \sum_{a,b \in A} \Big|\tr( p^x_a h^{\frac 1 2} p^y_b h^{\frac 1 2}) - \tau( q_\xi p^x_a q_\xi p^y_b q_\xi)\Big| d\nu(x,y) \leq 4\sqrt{2} \delta^{\frac 1 4}.
\end{equation}
To prove \eqref{eq:commutator_p_h}, we expand
\[ \int \sum_a \| [p^x_a,h_\xi^{\frac 1 2}]\|_{L_2(\cM)}^2 d\mu(x) = 2-2 \int \sum_a \|h_\xi^{\frac 1 4} p^x_a h_\xi^{\frac 1 4}\|_{L_2(\cM)}^2 d\mu(x).\]
By \eqref{eq:standardform}, $\int \sum_{a \in A} \langle p^x_a \tilde{p}^x_a \xi,\xi \rangle d\mu(x)$ is the scalar product in  $L_2(\mu;\ell_2(A;L_2(\cM))$ of $(x,a) \mapsto h_\xi^{\frac 1 4} p^x_a h_\xi^{\frac 1 4}$ with $(x,a) \mapsto h_\xi^{\frac 1 4} p'^x_a h_\xi^{\frac 1 4}$. These two vectors clearly have norm at most $1$, so by the Cauchy-Schwarz inequality and our assumption, they both have norm at least $1-\delta$. In particular,
\[ \int \sum_a \|h_\xi^{\frac 1 4} p^x_a h_\xi^{\frac 1 4}\|_{L_2(\cM)}^2 d\mu(x) \geq (1-\delta)^2 \geq 1-2\delta\]
and \eqref{eq:commutator_p_h} follows.
    
The next observation is that
\begin{equation}\label{eq:pandpprime_close} \int \sum_a \|h_\xi^{\frac 1 2}(p^x_a - (p'^x_a)^{\frac 1 2})\|_{L_2(\cM)}^2 d\mu(x)\leq 6 \sqrt{\delta}.
\end{equation}
Expanding the square, we can write the left-hand side of \eqref{eq:pandpprime_close} as
\[ 
\int \big(2 -2 \sum_a \tr(h_\xi^{\frac 1 2} p^x_a h_\xi^{\frac 1 2} (p'^x_a)^{\frac 1 2}) + 2 \Re \sum_a \tr(h_\xi^{\frac 1 2} [p^x_a, h_\xi^{\frac 1 2}] (p'^x_a)^{\frac 1 2})\big)d\mu(x).\]
Using our assumption and the Cauchy-Schwarz inequality, we obtain that the left-hand side of \eqref{eq:pandpprime_close} is less than
\[ 2\delta + 2 \Big( \int \sum_a \| [p^x_a,h_\xi^{\frac 1 2}]\|_{L_2(\cM)}^2 d\mu(x)\Big)^{\frac 1 2} \leq 2\delta + 4 \sqrt{\delta}.\]
The last inequality is by \eqref{eq:commutator_p_h}. This proves \eqref{eq:pandpprime_close}.

Now observe that $\tr(p^x_a h_\xi^{\frac 1 2} p^y_b h_\xi^{\frac 1 2})$ is the norm of $p^x_a h_\xi^{\frac 1 2} p^y_b$ in $L_2(\cM)$. Similarly, by \eqref{eq:standardform}, $\langle p^x_a \tilde{p}^y_b \xi,\xi\rangle$ is the norm of $p^x_a h_\xi^{\frac 1 2} (p'^y_b)^{\frac 1 2}$. We can therefore apply Lemma~\ref{lem:Hilbert_space} and obtain 
\[ \int \sum_{a,b} \Big| \langle p^x_a \tilde{p}^y_b \xi,\xi\rangle - \tr(p^x_a h_\xi^{\frac 1 2} p^y_b h_\xi^{\frac 1 2})| d\nu(x,y) \leq 2\Big(\int \sum_{a,b}\|p^x_a h_\xi^{\frac 1 2} (p^y_b-(p'^y_b)^{\frac 1 2})\|_{L_2(\cM)}^2 d\nu(x,y) \Big)^{\frac 1 2}, \]
which is less than $3\delta^{\frac 1 4}$ by \eqref{eq:pandpprime_close}. Together with \eqref{eq:correlations_symmetric}, this concludes the proof of the first half of the theorem. The second half is immediate from the moreover part of Lemma~\ref{lem:symmetric_correlation_close_to_tracial}.
  \end{proof}

  \section{Application for strategies and games}\label{sec:games}
  We now prove the main results of the introduction, Theorem~\ref{thm:main_games} and~\ref{thm:main_gamesvN}. As announced, we prove a more precise form of them in Corollary~\ref{cor:application_games}. Before, we introduce some notation. Given a von Neumann algebra $\cM \subset B(\cH)$, we denote by $\val_\cM(\cG)$ the supremum of $\val(\cG,\cP)$ over all commuting strategies with $p^x_a \in \cM$ and $\tilde{p}^y_b \in \cM'$.

  If $(\cM,\tau)$ is a von Neumann algebra with a normal tracial state $\tau$, the \defi{tracial (or synchronous) value}  $\val_{\cM,\tau}^{tr}(\cG)$ is the supremum of
  \[ \int \sum_{a,b} D(x,y,a,b) \tau( p^x_a p^y_b),\]
  over all partitions of unity $(p^x_a)_{a \in A} \subset \cM$.

Finally, if $(\cM,\tau)$ is a von Neumann algebra with a semifinite normal trace, we denote by $\val_{\leq \cM,\tau}^{tr}(\cG)$ the supremum of $\val_{q \cM q, \frac{1}{\tau(q)}\tau}^{tr}(\cG)$ over all nonzero projections $q \in \cM$ with finite trace.

  \begin{corollary}\label{cor:application_games} Let $\cG$ be an $\alpha$-synchronous game and $\cM\subset B(\cH)$ be a von Neumann algebra. If $\val_\cM(\cG) \geq 1-\delta$, then there is a trace one projection $q \in c(\cM)$ such that $\val_{q c(\cM)q,\tau}^{tr}(\cG) \geq 1 - c (\delta/\alpha)^{\frac 1 4}$ for a universal constant $c$. 
  \end{corollary}
  \begin{proof}
Assume that $\val_{\cM}(\cG) > 1-\delta$. This means that there is a commuting strategy $\cP = (\cH,\xi,p,\tilde{p})$ with $p \subset \cM$ and $\tilde{p}\subset \cM'$ such that $\val(\cG,\cP) \geq 1-\delta$, that is
    \[ \int \sum_{a,b} (1-D(x,y,a,b)) \langle p^x_a \tilde{p}^y_b \xi,\xi\rangle d\nu(x,y) \leq \delta.\] In particular, restricting the integral to the diagonal $\{(x,x) \mid x\in X\}$ and using that $\cG$ is $\alpha$-synchronous, we obtain
    \[ \alpha \int \sum_{a\neq b} \langle p^x_a \tilde{p}^x_b \xi,\xi\rangle d\mu(x) \leq \delta,\]
    where $\mu$ is the first marginal of $\nu$. Equivalently,
    \[\int \sum_{a} \langle p^x_a \tilde{p}^x_a \xi,\xi\rangle d\mu(x) \geq 1 - \frac{\delta}{\alpha}.\]
    So it follows from Theorem~\ref{thm:almost_synchronous_correlations} that the correlation $\langle p^x_a \tilde{p}^y_b \xi,\xi\rangle$ is $O( (\delta/\alpha)^{\frac 1 4})$ close in $L_1$-norm to the correlation $\tau( r^x_a r^y_b)$. In particular, the value of the game $\cG$ at these strategies are $O( (\delta/\alpha)^{\frac 1 4})$-close, and the corollary follows.
    \end{proof}
  In the particular case when $\cM$ is semifinite, we obtain.
  \begin{corollary}\label{cor:application_games_semifinite} Let $\cG$ be an $\alpha$-synchronous game and $\cM\subset B(\cH)$ be a von Neumann algebra with a normal faithful tracial state $\tau_0$. If $\val_\cM(\cG) \geq 1-\delta$, then $\val_{\leq \cM,\tau_0}^{\tr}(\cG) \geq 1 - c (\delta/\alpha)^{\frac 1 4}$.
  \end{corollary}
  This corollary is the combination of Corollary~\ref{cor:application_games} with the following lemma.
  \begin{lemma} Let $\cG$ be any symmetric game and $(\cM,\tau_0)$ be von Neumann with a normal faithful semifinite trace. Then $\val_{\leq \cM,\tau_0}^{\tr}(\cG) = \val_{\leq c(\cM),\tau}^{\tr}(\cG)$.
  \end{lemma}
  \begin{proof} By Example~\ref{example:core_semifinite}, $(c(\cM),\tau)$ is isomorphic to $(L_\infty(\R) \overline{\otimes} \cM,e^{-s} ds \otimes \tau_0)$, so a finite projection $q \in c(\cM)$ corresponds to a measurable family $(q_s)_{s \in \R}$ of finite projections satisfying $\int \tau_0(q_s) e^{-s}ds=\tau(q)$. Moreover, if $(r^x_a)_{a \in A} \subset q c(\cM) q$ is a family of PVMs, the correlation $\big(\frac{1}{\tau(q)}\tau(r^x_a r^y_b)\big)_{x,y,a,b}$ is a convex combinations of the correlations $\big(\frac{1}{\tau_0(q_s)} \tau_0(r^x_a(s) r^y_b(s))\big)_{x,y,a,b}$. The lemma follows.
    \end{proof}
  When $\cM = M_d(\C)$, we recover Vidick's theorem.
  \begin{corollary}\cite{MR4373849} If an $\alpha$-synchronous game $\cG$ admits a commuting strategy with value $\geq 1-\delta$ on a Hilbert space of finite dimension $\leq d^2$, then $\val_{\leq M_d(\C),Tr}^{\tr}(\cG) \geq 1 - c (\delta/\alpha)^{\frac 1 4}$.
\end{corollary}
  \begin{proof} Let $\cP$ be a strategy on $\cH$ of dimension $\leq d^2$ with valued $\geq 1-\delta$. Consider the algebras $\cM$ and $\widetilde{\cM}$  generated by $\{p^x_a \mid x \in X,a\in A\}$ and $\{\tilde{p}^x_a \mid x \in X,a\in A\}$. Without loss of generality we can assume that $\dim \cM \leq \dim \widetilde{\cM}$. They are commuting sub-algebras of $B(\cH)$, so $\cM \otimes \widetilde \cM$ embeds into $B(\cH)$. Taking dimensions, we obtain $(\dim \cM)^2 \leq \dim \cM \dim{\widetilde \cM} \leq d^4$. Therefore, $\dim \cM \leq d^2$ and $\dim \cM$ is isomorphic to a direct sum of algebras $M_k(\C)$ with $k \leq d$ and we deduce that $\val_{M_k(\C)}(\cG) \geq 1-\delta$ for some $k\leq d$. We conclude by Corollary~\ref{cor:application_games_semifinite}. 
  \end{proof}
  \bibliographystyle{alpha}
\bibliography{biblio}

\end{document}